\documentclass[11pt]{article}
\usepackage{amssymb,amsmath,amsthm,stmaryrd}
\usepackage{enumerate}
\usepackage{pgf,tikz}
\usetikzlibrary{arrows}
\usepackage[labelsep=space]{caption}
\usepackage{verbatim}

\newtheorem{theorem}{Theorem}[section]
\newtheorem{corollary}[theorem]{Corollary}

\newtheorem{proposition}[theorem]{Proposition}

\newtheorem{definition}[theorem]{Definition}
\newtheorem{example}[theorem]{Example}
\numberwithin{equation}{section}

\newcommand{\R}{{\bf{R}}}
\newcommand{\U}{{\bf{U}}}
\newcommand{\D}{{\bf{D}}}
\newcommand{\RR}{{\bf{\overline{R}}}}
\newcommand{\UU}{{\bf{\overline{U}}}}
\newcommand{\bc}{{\bf{c}}}
\newcommand{\bd}{{\bf{d}}}
\newcommand{\ba}{{\bf{a}}}
\newcommand{\bb}{{\bf{b}}}
\newcommand{\wt}{\operatorname{wt}}
\newcommand{\pyr}{\operatorname{Pyr}}
\newcommand{\prism}{\operatorname{Prism}}
\newcommand{\zab}{\mathbb{Z}\langle\ba,\bb\rangle}
\newcommand{\zcd}{\mathbb{Z}\langle\bc,\bd\rangle}
\newcommand\Tstrut{\rule{0pt}{2.6ex}}

\newcommand{\vanish}[1]{}

\begin{document}

\author{N. Bradley Fox}
\title{A Lattice Path Interpretation of the Diamond Product}
\date{}

\maketitle

\begin{abstract}
The diamond product is the poset operation that when applied to the face lattices of two polytopes results in the face lattice of the Cartesian product of the polytopes. Application of the diamond product to two Eulerian posets is a bilinear operation on the $\bc\bd$-indices of the two posets, yielding a product on $\bc\bd$-polynomials.   A lattice path interpretation is provided for this product of two $\bc\bd$-monomials.
\end{abstract}

\section{Introduction}
The $\bc\bd$-index is a polynomial in the non-commutative variables $\bc$ and $\bd$ that efficiently encodes an Eulerian poset's flag $f$-vector, which contains information on the number of chains through each set of ranks of the poset.  One primary example of an Eulerian poset is the face lattice of a convex polytope.  An important characteristic of the $\bc\bd$-index is that it is a useful invariant for computations, as explicit formulas have been developed to calculate the effect that poset and polytope operations have on the $\bc\bd$-index.  Polytope operations that have been studied include the prism, pyramid, free join, Cartesian product, and truncation of a vertex, as well as their associated poset operations.

Ehrenborg and Readdy used coproducts and derivations in \cite{Ehrenborg_Readdy} to generate expressions for the $\bc\bd$-index of polytopes under operations such as the prism of a polytope, or more generally the Cartesian product of polytopes.  The equivalent poset operation to this product is the diamond product.  The expressions that were developed, unfortunately, were rather complicated and required the use of auxiliary variables $\ba$ and $\bb$. Ehrenborg and Fox \cite{Ehrenborg_Fox} (no relation to author) improved upon the earlier work by developing recursive formulas for the bilinear operator that corresponds to the diamond product.  

The diamond product operator is non-negative on $\bc\bd$-indices, thus leading to the study of combinatorial interpretations of the resulting coefficients.  Slone \cite{Slone} examined the specific case of the diamond product of two butterfly posets, whose $\bc\bd$-indices are simply powers of $\bc$.  He found that one can interpret the polynomial as a weighted sum of lattice paths.  In this paper, a generalization of Slone's lattice path interpretation is given for the diamond product of any two $\bc\bd$-monomials, in addition to a lattice path interpretation for the product of $\ba\bb$-monomials.

In Section 2 we introduce the $\bc\bd$-index of Eulerian Posets and its underlying coalgebra structure.  In Section~3 we introduce the diamond product of two posets.  The lattice path interpretation of this product is discussed in Section 4.  Finally, we state an open problem in Section 5.

\section{Posets, the $\bc\bd$-index, and coproducts}
Consider the poset $P$ to be a graded poset of rank $n+1$ with rank function $\rho$, unique minimum element~$\hat{0}$, and unique maximum element $\hat{1}$.  For further terminology on partially ordered sets, see \cite[Chapter~3]{Stanley}.  

A $\it{chain}$ $c$ in such a poset $P$ is a linearly ordered subset of $P$.  We will only consider chains that contain the minimum element $\hat{0}$ and the maximum element $\hat{1}$; hence, we write $c$ as $c=\{\hat{0}=x_0<x_1<\cdots <x_k=\hat{1}\}$.  Let $S$ be a subset of the set $\{1,2,\ldots,n\}$, and define $f_S(P)=f_S$ to be the number of chains in the poset $P$ whose elements $x_1,\ldots, x_{k-1}$ have ranks that are exactly the elements of the set~$S$.  The $2^n$ values of $f_S$ are collectively known as the $flag$ $f$-$vector$ of $P$. The $flag$ $h$-$vector$ is defined using the relation 
$$h_S=\sum_{T\subseteq S} (-1)^{|S-T|}\cdot f_T,$$
which is equivalent to $$f_S=\sum_{T\subseteq S} h_T.$$

Let $\ba$ and $\bb$ be non-commutative variables.  For a subset $S$ of $\{1,\ldots,n\}$, define the $\ba\bb$-monomial $u_S=u_1\cdots u_n$ in which $u_i=\ba$ if $i\notin S$ and $u_i=\bb$ if $i\in S$.  Define the $\ba\bb$-$index$ $\Psi(P)$ of the poset $P$ to be the $\ba\bb$-polynomial 
$$\Psi(P)=\sum_S h_S\cdot u_S,$$
where $S$ ranges over all subsets of $\{1,\ldots, n\}$.

The M\"{o}bius function $\mu$ of a poset $P$ is defined by $\mu(x,x)=1$ and the recursion  $\mu(x,y)=-\sum_{x\leq z<y} \mu(x,z)$ for $x<y$.  A poset $P$ is \textit{Eulerian} if its M\"{o}bius function satisfies the relation $\mu(x,y)=(-1)^{\rho(y)-\rho(x)}$ for all intervals $[x,y]$ in $P$.  A key example of Eulerian posets is the face lattice of a convex polytope.  The following result was conjectured by Fine and later proved by Bayer and Klapper \cite{Bayer_Klapper}.

\begin{theorem}[Bayer--Klapper]
\label{cdindex}
The $\ba\bb$-index $\Psi(P)$ of an Eulerian poset $P$ is a non-commutative polynomial in $\bc=\ba+\bb$ and $\bd=\ba\cdot \bb+\bb\cdot \ba$.
\end{theorem}

When written in terms of $\bc$ and $\bd$, we call $\Psi(P)$ the $\bc\bd$-index of the poset $P$, although the same notation is used for the $\ba\bb$-index and the $\bc\bd$-index.  The existence of the $\bc\bd$-index is equivalent to the fact that the flag $f$-vector of an Eulerian poset satisfies the generalized Dehn--Sommerville relations, due to Bayer and Billera in \cite{Bayer_Billera}.   For examples and more information on the $\bc\bd$-index of posets, see \cite[Section~3.17]{Stanley}.

We now briefly discuss the coalgebraic structures of the $\ba\bb$-index and the $\bc\bd$-index that were introduced in~\cite{Ehrenborg_Readdy}.   Given an abelian group $V$, a \textit{coproduct} is a linear map $\Delta: V\longrightarrow V\otimes V$.  We will use Sweedler notation to denote the coproduct of an element $v\in V$ as $\Delta(v)=\sum_v v_{(1)}\otimes v_{(2)}$, where this sum is over finitely many pairs $v_{(1)}$ and $v_{(2)}$.  An abelian group $V$ with associative product $\cdot$ and coassociative coproduct $\Delta$ is called a \textit{Newtonian coalgebra} if it satisfies the following identity
$$\Delta(u\cdot v)=\sum_u u_{(1)} \otimes u_{(2)} \cdot v+\sum_v u\cdot v_{(1)} \otimes v_{(2)}.$$
It is straightforward to verify that the two coalgebras described below are both Newtonian.

First, let $\zab$ denote the polynomial ring in the non-commutative variables $\ba$ and $\bb$, where the degree of each variable is one.  For an $\ba\bb$-monomial $u=u_1\cdots u_n$, define
$$\Delta(u)=\sum_{i=1}^n u_1\cdots u_{i-1} \otimes u_{i+1}\cdots u_n,$$
and extend linearly to $\zab$.  As examples, $\Delta(1)=0$ and $\Delta(\ba)=\Delta(\bb)=1\otimes 1$.

Next, consider the subring $\zcd$ of $\zab$ generated by the variables $\bc$ and $\bd$ as defined in Theorem~\ref{cdindex}.  Once one calculates $\Delta(\bc)=\Delta(\ba+\bb)=2\cdot 1\otimes 1$ and $\Delta(\bd)=\Delta(\ba\cdot \bb+\bb\cdot \ba)=1\otimes \bc+\bc \otimes 1$, it can be verified that $\zcd$ is also a Newtonian coalgebra.

We now define two linear operators on these coalgebras.  Let $G:\zab \longrightarrow \zab$ be the derivation given by the rules $G(\ba)=\bb\cdot\ba$, $G(\bb)=\ba\cdot\bb$, and the product rule $G(u\cdot v)=G(u)\cdot~v+~u\cdot~G(v)$.  Since $G(\bc)=\bd$ and $G(\bd)=\bc\cdot\bd$, $G$ becomes a linear operator on $\zcd$ as well.  Then let $\pyr:~\zcd \longrightarrow \zcd$ be the linear operator defined by $\pyr(u)=u\cdot \bc+G(u)$. 

\section{The Diamond Product of Posets}
Given two graded posets $P$ and $Q$, we define the \textit{Cartesian product} of $P$ and $Q$ to be the poset $P\times Q=\{(x,y):x\in P,y\in Q\}$ with the order relation given by $(x,y)\leq_{P\times Q} (w,z)$ if $x\leq_P w$ and $y\leq_Q z$.  Using this product, we can then define the \textit{diamond product} of $P$ and $Q$ as the graded poset $P\diamond Q=(P-\{\hat{0}\})\times(Q-\{\hat{0}\})\cup \{\hat{0}\}$.  This product corresponds to the \textit{Cartesian product} of polytopes, defined as follows.  For an $m$-dimensional polytope $V$ and $n$-dimensional polytope $W$, we say the Cartesian product of $V$ and $W$ is the $(m+n)$-dimensional polytope $$V\times W=\{(x_1,\ldots, x_{m+n})\in \mathbb{R}^{m+n}\text{ : } (x_1,\ldots, x_m)\in V, (x_{m+1},\ldots, x_{m+n})\in W\}.$$
The connection between the diamond product and Cartesian product was noted by Kalai in \cite{Kalai}, where he stated that the face lattice of the Cartesian product of two polytopes corresponds to the diamond product of their face lattices, that is $\mathcal{L}(V\times W)=\mathcal{L}(V)\diamond \mathcal{L}(W).$  The diamond product specifically appears when studying the \textit{prism} of a polytope, defined as $\prism(V)=V\times I$, where $I$ is the unit interval.  As stated in Proposition 4.1 of \cite{Ehrenborg_Readdy}, $\mathcal{L}(\prism(V))=\mathcal{L}(V)\diamond B_2$.

Because of the importance of the prism operation and the Cartesian product in the study of polytopes, one needs to understand how these operations affect the $\ba\bb$- or $\bc\bd$-index of polytopes, or likewise their associated posets.  This leads to the investigation of the $\bc\bd$-index of the diamond product of two Eulerian posets.  Ehrenborg and Readdy \cite{Ehrenborg_Readdy} developed a bilinear operator for this purpose, as described in the following proposition.  One can find the precise definition along with additional properties and recurrences for this operator in Section 6 of \cite{Ehrenborg_Fox} and Section 10 of \cite{Ehrenborg_Readdy}.

\begin{proposition}[Ehrenborg--Readdy]
\label{diamond_prop}
There exists a bilinear operator from $\zcd\times \zcd$ to $\zcd$, denoted by $\diamond$, such that given any two Eulerian posets $P$ and $Q$, the $\bc\bd$-index of their diamond product is given by
$$\Psi(P\diamond Q)=\Psi(P)\diamond \Psi(Q).$$\\*
Hence for two polytopes $V$ and $W$, the $\bc\bd$-index of the Cartesian product $V\times W$ is given by 
$$\Psi(V\times W)=\Psi(V)\diamond \Psi(W).$$
\end{proposition} 

The bilinear operator described in Proposition~\ref{diamond_prop} is denoted as $N(u,v)$ in the papers \cite{Ehrenborg_Fox} and \cite{Ehrenborg_Readdy}, but we use the diamond product $u\diamond v$ to simplify the notation.  Also note that the diamond product operator can be extended to be a product of $\ba\bb$-polynomials instead of only $\bc\bd$-polynomials.

The following statements made by Ehrenborg and Fox in \cite{Ehrenborg_Fox} give useful properties and a recursive formula for calculating the diamond product of two $\ba\bb$- or $\bc\bd$-polynomials.  Proposition~\ref{diamond_recursion_ab} is a reformulated version of Proposition~7.6 of \cite{Ehrenborg_Fox}.  Likewise, Proposition~\ref{diamond_recursion} is a reformulation of Theorem~7.1 of \cite{Ehrenborg_Fox},  as was shown in Corollary~2.3.7 of \cite{Slone}.

\begin{corollary}
For any $\ba\bb$- or $\bc\bd$-polynomials $u$, $v$, and $w$, the following are satisfied 
\begin{align*}
u\diamond 1&= u,\\
u\diamond v&= v\diamond u,\\
u\diamond (v\diamond w)&=(u\diamond v)\diamond w.
\end{align*}
\end{corollary}

\begin{proposition}[Ehrenborg--Fox]
\label{diamond_recursion_ab}
For any $\ba\bb$-polynomials $u$ and $v$, the diamond product satisfies the following recursions:
\begin{align}
u\diamond (v\cdot \ba) &=(u\diamond v)\cdot \ba + \sum_u (u_{(1)} \diamond v)\cdot \ba\cdot\bb \cdot u_{(2)},\label{diamond1ab}\\
u\diamond (v\cdot \bb) &= (u\diamond v)\cdot \bb + \sum_u (u_{(1)} \diamond v)\cdot \bb\cdot\ba \cdot u_{(2)}\label{diamond2ab}.
\end{align}
\end{proposition}

\begin{proposition}[Ehrenborg--Fox]
\label{diamond_recursion}
For any $\bc\bd$-polynomials $u$ and $v$, the diamond product satisfies the following recursions:
\begin{align}
u\diamond (v\cdot \bc) &=(u\diamond v)\cdot \bc + \sum_u (u_{(1)} \diamond v)\cdot \bd \cdot u_{(2)},\label{diamond1}\\
u\diamond (v\cdot \bd) &= (u\diamond v)\cdot \bd + \sum_u (u_{(1)} \diamond v)\cdot \bd \cdot \pyr(u_{(2)})\label{diamond2}.
\end{align}
\end{proposition}

\section{Lattice Path Interpretation for $\ba\bb$-monomials}

Before introducing the lattice path interpretation for $\bc\bd$-monomials, we first introduce a similar interpretation for the diamond product of two $\ba\bb$-monomials.  Define the set of lattice paths $\Omega$ as words in the non-commutative letters $\D$, $\R$, and $\U$, where $\D$ is degree 2, and $\R$ and $\U$ are each degree 1.  The letters correspond to the lattice path steps as follows 
$$\text{Right}: \R=(1,0), \text{ Up}: \U=(0,1), \text{ and Diagonal}: \D=(1,1).$$
Let $\Omega(p,q)$ be the set of lattice paths using only these 3 steps from $(0,0)$ to $(p,q)$ which do not contain $\U\R$ as a contiguous subword, that is, as a factor.  

For a given pair of $\ba\bb$-monomials $u$ and $v$ with degrees $p$ and $q$, respectively, consider lattice paths in $\Omega(p,q)$ in which the axes are labeled by the words $u$ and $v$, as shown by the example in Figure~\ref{ex_ab_path}.  We now define a weight function for such paths based on this labeling.
\begin{figure}
\centering
\includegraphics[scale=.35]{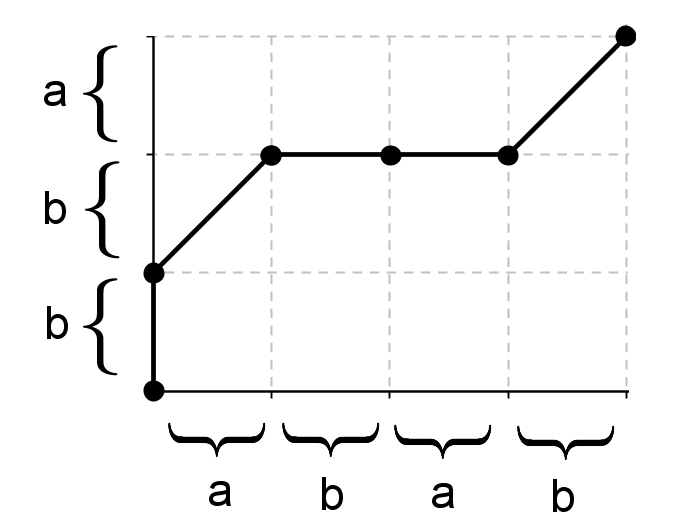}
\caption{: The lattice path $\U\D\R\R\D \in \Omega(4,3)$ labeled by the words $u=\ba\bb\ba\bb$ and $v=\bb\bb\ba$}
\label{ex_ab_path}
\end{figure}

\begin{definition}
For $p'\leq p$ and $q'\leq q$, define $\wt_{u,v}: \Omega(p',q') \longrightarrow \zab$ to be the multiplicative map, taking concatenation to be the product, determined by the following rules:
\begin{align*}
  \wt_{u,v}(\R)&=\begin{cases}
   \ba        & \text{if above an }\ba\text{ label}\\
   \bb		& \text{if above a }\bb\text{ label},
   \end{cases}\\
  \wt_{u,v}(\U)&= \begin{cases}
   \ba        & \text{if to the right of an }\ba\text{ label}\\
   \bb		& \text{if to the right of a }\bb\text{ label},
   \end{cases}\\
  \wt_{u,v}(\D)&= \begin{cases}
   \ba\cdot\bb        & \text{if to the right of an }\ba\text{ label}\\
   \bb\cdot\ba		& \text{if to the right of a }\bb\text{ label}.
   \end{cases} 
  \end{align*}
\end{definition}

For the example path in Figure~\ref{ex_ab_path}, we have $\wt_{\ba\bb\ba\bb,\bb\bb\ba}(\U\D\R\R\D)=\bb\bb\ba\bb\ba\ba\bb$. 

For a given $\ba\bb$-monomial~$u$ of degree $p$, define $\tau(u)\in \Omega(p,0)$ as the word $\tau(u)=\R^{\deg(u)}$.  Now that we have notation for creating horizontal paths, we give the interpretation for the diamond product of two $\ba\bb$-monomials as a sum of weighted lattice paths.

\begin{theorem}
For any two $\ba\bb$-monomials $u$ and $v$ of degree $p$ and $q$, respectively, the $\ba\bb$-polynomial $u\diamond v$ is given by the sum
$$u\diamond v=\sum_{P\in \Omega(p,q)} \wt_{u,v}(P).$$
\label{lattice_path__ab_thm}
\end{theorem}
\begin{proof}
To keep the notation simpler, we will leave out the dependency on $u$ and $v$ of the weight function.
The proof of this theorem is by induction on the degree $q$ of the monomial $v$.  For the base case, we assume that $q$ is 0, making $v=1$, and that the degree $p$ of $u$ is any non-negative integer.  The diamond product $u\diamond 1$ is $u$, and the only lattice path in $\Omega(p,0)$ is the horizontal path $\tau(u)$ of length~$p$.  The weight of this path $\tau(u)$ is $\wt(\tau(u))=u$ since it is only $\R$ steps along the labels of $u$; thus, the base case of the theorem is true.

Suppose the statement is true for any two words of degree $p'$ and $q'$ where $p'\leq p$ and $q'<q$.  We first assume that the last letter of $v$ is $\ba$, or $v=w\cdot \ba$.  According to Equation~\eqref{diamond1ab}, we have 
$$u\diamond (w\cdot \ba) =(u\diamond w)\cdot \ba + \sum_u (u_{(1)} \diamond w)\cdot \ba\cdot\bb \cdot u_{(2)}.$$

By induction, the first term is
\begin{equation}
(u\diamond w)\cdot \ba = \sum_{P\in \Omega(p,q-1)} \wt(P\cdot \U).
\label{ab_a}
\end{equation}
Since the final $\U$ step is to the right of an $\ba$ label, $\ba$ is the correct weight for this step.

For the terms that result from the coproduct, we observe that the cases of $u$ being broken apart by the coproduct at either an $\ba$ or a $\bb$ are identical.  We assume that either $u=y\cdot \ba \cdot z$ or $u=y\cdot \bb\cdot z$ where $y$ is of degree $i$.  Hence, we have $u_{(1)}\otimes u_{(2)}=y\otimes z$ in each case since $\Delta(\ba)=\Delta(\bb)=1\otimes 1$.  This gives the term
\begin{equation}
(y\diamond w)\cdot \ba\cdot\bb\cdot z=\sum_{P\in \Omega(i,q-1)} \wt(P\cdot \D\cdot \tau(z)).
\label{ab_b}
\end{equation}
Notice that the weight of a $\D$ step does not depend on the label below that step, rather it only depends on the label on the vertical axis.  Since this $\D$ step is to the right of the $\ba$ label that ends the word $v$, its weight is $\ba\cdot\bb$, which matches the left side of the equation.

Since we only consider lattice paths without consecutive $\U\R$ steps, every lattice path in $\Omega(p,q)$ must end in a $\U$ step or end in a $\D$ step followed by a horizontal path.  The paths contained within equation~\eqref{ab_a} correspond to the paths ending in $\U$, and the remaining possible paths with the $\D$ step are found in equation~\eqref{ab_b}.  Thus $\Omega(p,q)$ decomposes into a disjoint union of lattice paths as
\begin{align*}
\Omega(p,q)=&\{P\cdot \U: P\in \Omega(p,q-1)\}\\
&\dot\cup \{P\cdot \D\cdot \tau(z): P\in \Omega(i,q-1), u=y\cdot \ba\cdot z \text{ or } u=y\cdot\bb\cdot z\},
\end{align*}
completing the proof if $v$ ends with the letter $\ba$.

If we instead assume that $v=w\cdot \bb$, then Equation~\eqref{diamond2ab} gives us
$$u\diamond (w\cdot \bb)=(u\diamond w)\cdot \bb+\sum_u (u_{(1)}\diamond w)\cdot \bb\cdot\ba \cdot u_{(2)}.$$
This second situation follows nearly identically to the first case from this point.  This is because the lattice paths ending in $\U$ would have $\bb$ as the weight for this final step since it would be to the right of a $\bb$ label.  Additionally, the $\D$ step in lattices paths ending in a $\D$ step followed by a horizontal path will contribute a weight of $\bb\cdot\ba$ since this step will also be to the right of the final $\bb$ label.  This second case concludes the proof of the theorem.

\end{proof}

\section{Lattice Path Interpretation for $\bc\bd$-monomials}

To try to give a better understanding of the recursive formulas given in~\eqref{diamond1} and~\eqref{diamond2} that Ehrenborg and Fox developed for the diamond product of two $\bc\bd$-polynomials, Slone examined in \cite{Slone} the specific case of the diamond product of the form $\bc^p\diamond \bc^q$.  He was able to interpret the coefficients of the resulting $\bc\bd$-polynomial using weighted lattice paths.  

Concentrating on the diamond product of powers of $\bc$, or $\bc^p\diamond\bc^q$, Slone defined the set of lattice paths $\Lambda$ as words in the non-commutative letters $\D$, $\R$, and $\U$, in which $\D$ has degree 2 whereas $\R$ and $\U$ both have degree 1.  As defined in the $\ba\bb$-index case, these letters correspond to lattice path steps as follows

$$\text{Right}: \R=(1,0), \text{ Up}: \U=(0,1), \text{ and Diagonal}: \D=(1,1).$$

Let $\Lambda(p,q)$ be the set of lattice paths using only these 3 steps from $(0,0)$ to $(p,q)$ which do not contain $\U\R$ as a contiguous subword.  Note that labeling the axes, as was done in the $\ba\bb$-index case, is not necessary here since each letter in the $\bc\bd$-monomials is a $\bc$. Define $\wt: \Lambda(p,q) \longrightarrow \zcd$ to be the multiplicative map, taking concatenation to be the product, determined by $\wt(\D)=2\bd$ and $\wt(\R)=\wt(\U)=\bc.$  The main result of Slone's work on the diamond product is the following statement, which is Proposition 2.4.2 in \cite{Slone}.

\begin{proposition}[Slone]
For any non-negative integers $p$ and $q$, the $\bc\bd$-polynomial $\bc^p \diamond \bc^q$ is given by the~sum
$$\bc^p \diamond \bc^q=\sum_{P\in \Lambda(p,q)} \wt(P).$$
\end{proposition}

Now we extend Slone's interpretation to look beyond the case of $\bc\bd$-monomials consisting of powers of $\bc$ to the diamond product of any two $\bc\bd$-monomials. Define the set of lattice paths $\Gamma$ as words in the noncommutative letters $\R$, $\U$, $\D$, $\RR$, and $\UU$.  We consider $\R$ and $\U$ to be degree $1$, and $\D$, $\RR$, and $\UU$ to be degree $2$.  The letters correspond to the steps

$$\text{Right: } \R=(1,0), \text{ Up: } \U=(0,1), \text{ Diagonal: } \D=(1,1),$$
$$\text{Double Right: } \overline{\bf{R}}=(2,0), \text{ and Double Up: }\overline{\bf{U}}=(0,2).$$
Let $\Gamma(p,q)$ be the set of all lattice paths from the origin to $(p,q)$ using the 5 steps described above and which do not contain consecutive $\U\R$, $\U\RR$, $\UU\R$, or $\UU\,\RR$ steps.

We now restrict this set to a particular subset $\Gamma(u,v)$ given two $\bc\bd$-monomials $u$ and $v$ with the degrees of the monomials being $p$ and $q$, respectively.  This subset within $\Gamma(p,q)$ requires that the word and its corresponding lattice path adhere to the following four rules, where we label the horizontal axis by the word $u$ and likewise label the vertical axis  by $v$, as shown in Figure~\ref{ex_path}.  This is similar to the labels used earlier with $\ba\bb$-monomials except that the $\bd$ label covers two units on the axis.  In the example, we have $u=\bf{ddcc}$ and $v=\bf{cdc}$; hence, the degrees are $p=6$ and $q=4$, with the lattice path $\D\R\RR\D\D\U$ being shown.

\begin{figure}
\centering
\includegraphics[scale=.35]{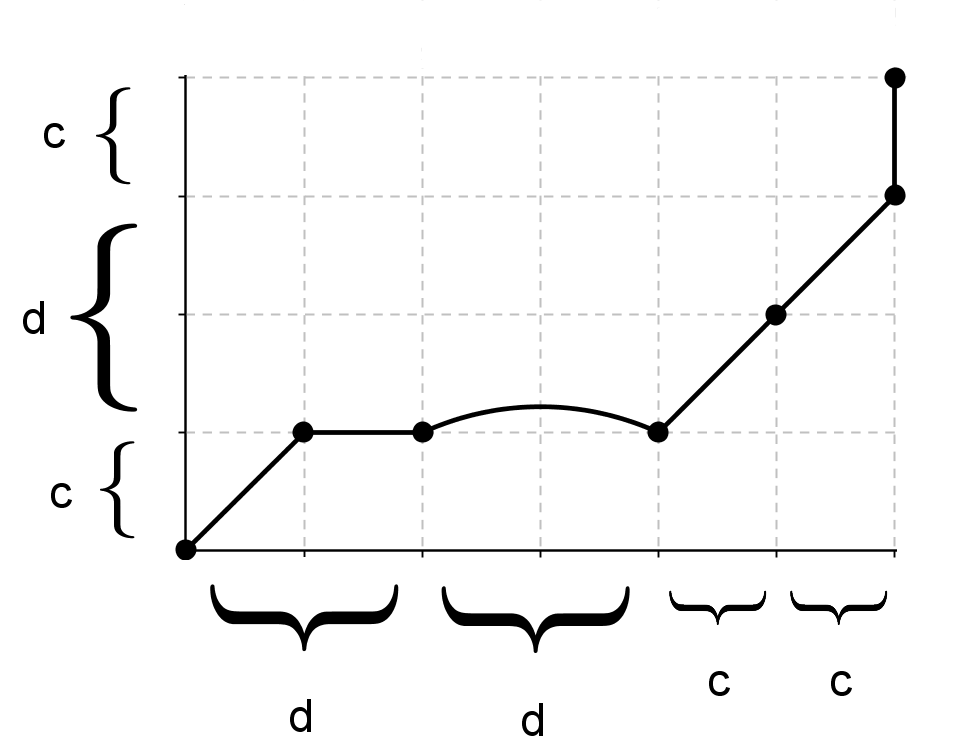}
\caption{: The lattice path $\D\R\RR\D\D\U \in \Gamma(\bd\bd\bc\bc,\bc\bd\bc)$}
\label{ex_path}
\end{figure}

The rules for a word $P\in \Gamma(p,q)$ to be in $\Gamma(u,v)$ are as follows:

\begin{enumerate}
\item No $\U$ step is allowed at the bottom of a $\bd$ label on the vertical axis.

\item Although an $\R$ step is allowed along the first part of a $\bd$ label on the horizontal axis, two consecutive $\R$ steps along such a $\bd$ label are not allowed.

\item A $\UU$ step is only allowed at the bottom of a $\bd$ label on the vertical axis, and similarly, an $\RR$ step is only allowed at the left of a $\bd$ label on the horizontal axis.

\item If a $\D$ step is at the bottom of a $\bd$ label on the vertical axis, then the steps $\D\R$ above a $\bd$ label on the horizontal axis and within the top half of this $\bd$ label on the vertical axis are not allowed.

\end{enumerate}

\begin{example} {\rm One can compute the diamond product of $\bf{cd}$ and $\bf{dc}$ as }
$$\bf{cd}\diamond\bf{dc}=3\bf{cddc}+\bf{ccdcc}+\bf{ccdd}+\bf{cdccc}+2\bf{cdcd}+2\bf{ddcc}+4\bf{dcdc}+2\bf{dccd}+4\bf{ddd}.$$
{\rm There are $13$ lattice paths in $\Gamma(\bf{cd},\bf{dc})$, which are shown in Figure~\ref{example}.  Note that none of the paths begin with $\U$ as required by rule 1 since the word $\bf{dc}$ begins with $\bd$. Additionally, due to rule 4, the path $\D\D\R\U$ is omitted.  The terms of $\bf{cd}\diamond\bf{dc}$ can be obtained from the lattice paths by weighting each $\R$ and $\U$ step by $\bc$, each $\RR$ and $\UU$ step by $\bd$, and each $\D$ step by $\bd$ if it is above a $\bd$ label or by~$2\bd$ if it is above a $\bc$ label, with the exception of making the coefficient $2$ for the lattice path $\R\D\U\D$.  Some of the paths, such as $\D\R\D\U$ and $\D\U\D\R$, give the same term of $\bf{cd}\diamond\bf{dc}$, leading to only $9$ terms from the $13$ lattice paths.  The paths and their corresponding weights are given in Table 1.}
\end{example}

\begin{figure}[t]
  \centering
    \includegraphics[scale=.8]{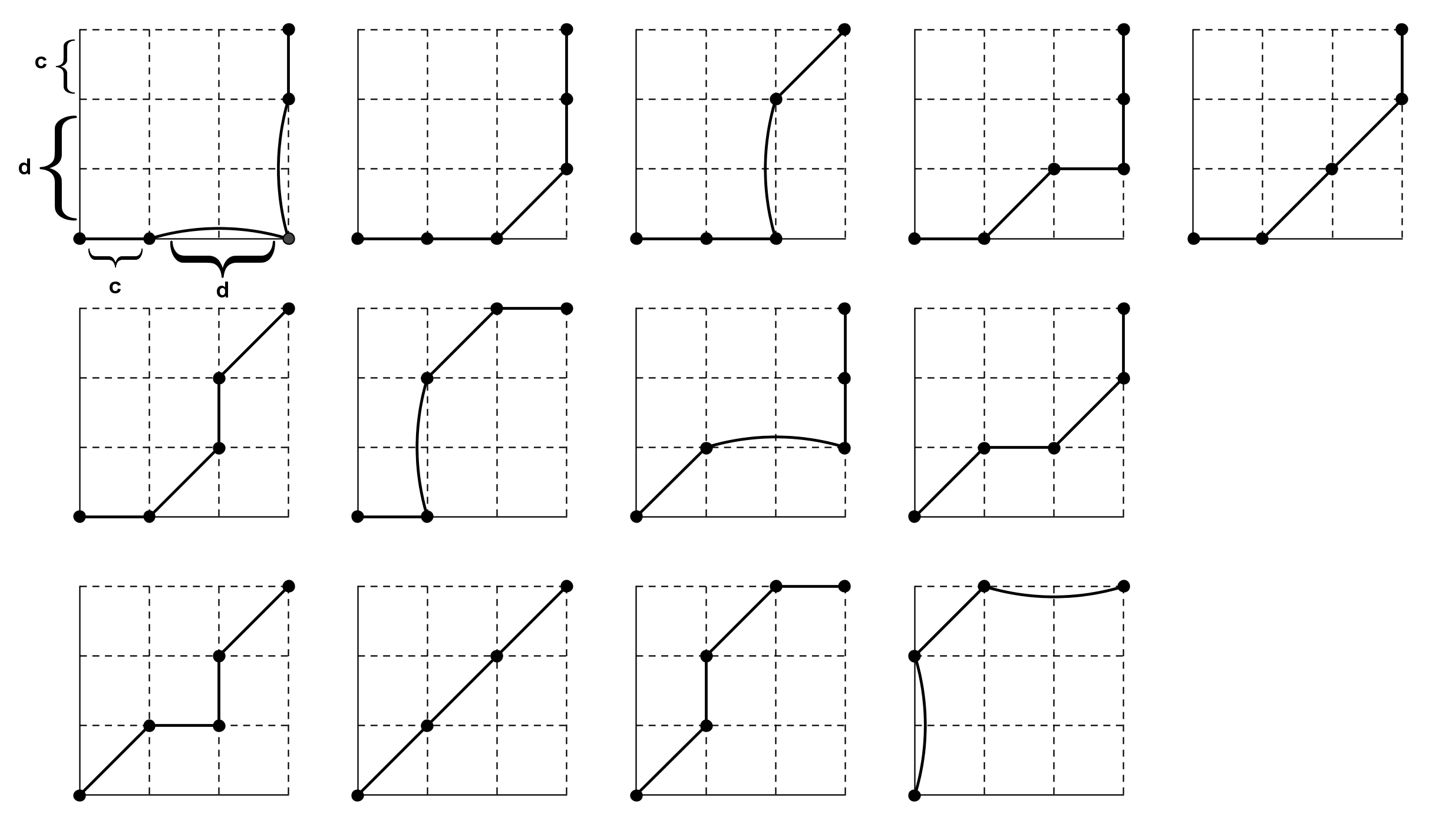}
  \caption{: The lattice paths in $\Gamma(\bc\bd,\bd\bc)$}
  \label{example}
\end{figure}
\begin{table}[h]
\label{table:table 1}
\centering
{\renewcommand{\arraystretch}{1.2}}
\begin{tabular}{c | c c c c c }

Path & $\R\RR\,\UU\U$ & $\R\R\D\U\U$ & $\R\R\UU\D$ & $\R\D\R\U\U$ & $\R\D\D\U$ \\ \hline
Weight & $\bc\bd\bd\bc$ & $\bc\bc\bd\bc\bc$ & $\bc\bc\bd\bd$ & $\bc\bd\bc\bc\bc$ & $\bc\bd\bd\bc$ \vspace{.5cm} \\  

Path &  $\R\D\U\D$ &  $\R\UU\D\R$ &  $\D\RR\U\U$ &  $\D\R\D\U$ & \\ \hline
Weight & $2\bc\bd\bc\bd$ & $\bc\bd\bd\bc$ & $2\bd\bd\bc\bc$ & $2\bd\bc\bd\bc$  \vspace{.5cm} \\  

Path & $\D\R\U\D$ & $\D\D\D$ &  $\D\U\D\R$ & $\UU\D\RR$ & \\ \hline
Weight & $2\bd\bc\bc\bd$ & $2\bd\bd\bd$ & $2\bd\bc\bd\bc$ & $2\bd\bd\bd$  \\
\end{tabular}
\caption{: The weights of the lattice paths in $\Gamma(\bc\bd,\bd\bc)$}
\end{table}

\bigskip 

The following definition gives the method of weighting the steps of the lattice paths in $\Gamma(u,v)$ for generic words $u$ and $v$ to obtain the $\bf{cd}$-index of the diamond product, albeit the choice of coefficient for weight of the $\D$ steps becomes complicated, explaining the need for the exception in the previous example.

\begin{definition}
For $u'$ an initial subword of $u$, that is, $u$ can be factored as $u=u'\cdot u''$, and $v'$ an initial subword of $v$, define $\wt_{u,v}:\Gamma(u',v') \longrightarrow \zcd$ to be the multiplicative map determined by

$$
\wt_{u,v}(\R)=\wt_{u,v}(\U)=\bc, \hspace{.5cm}
\wt_{u,v}(\RR)=\wt_{u,v}(\UU)=\bd, \hspace{.5cm}
\wt_{u,v}(\D)=k\bd,
$$
where depending on the location of a diagonal step $\D$, the scalar $k$ is given by

 $k =
  \begin{cases}
   2        & \text{if above a }\bc\text{ label and to the right of either a }\bc\text{ label or the bottom of a }\bd\text{ label}\\
   2		& \text{if above the first part of a }\bd\text{ label, to the right of a }\bc\text{ label, and followed by a }\U\text{ step,}\\
   			& \text{a }\UU\text{ step, or a }\D\text{ step}\\
   2		&  \text{if above the first part of a }\bd\text{ label, to the right of the bottom of a }\bd\text{ label, }\\
   			&  \text{and followed by a }\U\text{ step}\\  
   1        & \text{otherwise}.
  \end{cases}$
\end{definition} 
\noindent Note that this weight function matches Slone's weight function when we restrict our view to lattice paths in $\Gamma(\bc^p,\bc^q)=\Lambda(p,q)$, because the coefficient of a $\D$ step will always be 2 in this situation.

With the weight function being formally defined, we can now state the main result, but we first define a map to create horizontal paths that will be useful in its proof, as was done with the map $\tau$ in the $\ba\bb$-index case. Define $\pi$ such that for a given $\bc\bd$-monomial~$u$, $\pi(u)$ is the word in $\Gamma(u,1)$ resulting from replacing each $\bc$ in $u$ with the step $\R$ and each $\bd$ with the step~$\RR$.  
This map will be important in the proof of Theorem~\ref{lattice_path_thm} since rules $2$ and $3$ imply that $\pi(u)$ is the only valid horizontal path along a portion of the horizontal axis labeled by $u$.

\begin{theorem}
For any two $\bc\bd$-monomials $u$ and $v$, the $\bc\bd$-polynomial $u\diamond v$ is given by the sum
$$u\diamond v=\sum_{P\in \Gamma(u,v)} \wt_{u,v}(P).$$
\label{lattice_path_thm}
\end{theorem}
\begin{proof}

Again to simplify notation, the dependency of the weight function on the words $u$ and $v$ will be omitted. We will prove this result using induction on the degree $q$ of $v$.  For the base case when $q=0$ and the degree of $u$ is any nonnegative integer $p$, we have that $v=1$.  The diamond product $u \diamond 1$ is simply $u$, and the only lattice path in $\Gamma(u,1)$ is $\pi(u)$, the horizontal path along the labels from $u$.  The fact that $\wt(\pi(u))=u$ shows that the base case is true.  

Suppose the statement is true for any two words of degree $p'$ and $q'$ where $p'\leq p$ and $q'<q$.  We will break up the proof for $u\diamond v$ according to the final letter of $v$.

\textbf{Case 1}: Assume $v=w\cdot \bc$.  Due to equation~\eqref{diamond1}, we have
$$u\diamond (w\cdot \bc)=(u\diamond w)\cdot \bc+\sum_u (u_{(1)}\diamond w)\cdot \bd\cdot u_{(2)}.$$

By induction, the first term is
\begin{equation}
(u\diamond w)\cdot \bc=\sum_{P\in \Gamma(u,w)} \wt(P\cdot \U).
\label{a}
\end{equation}
An illustration of the lattice paths in equation~\eqref{a} as well as the next equation can be seen in Figure~\ref{proof_example1}.

For the remaining terms that result from the coproduct, we must separately examine the cases of~$u$ being broken apart by the coproduct at either a $\bc$ or $\bd$.  If broken up at a $\bc$, we assume $u=y\cdot \bc\cdot z$; thus, $u$ splits such that $u_{(1)}\otimes u_{(2)}=2y\otimes z$.  This gives the term
\begin{equation}
(y\diamond w)\cdot 2\bd\cdot z=\sum_{P\in \Gamma(y,w)} \wt(P\cdot \D\cdot \pi(z)).
\label{b}
\end{equation}
Since the $\D$ step is above the $\bc$ label that is between $y$ and $z$ and to the right of a $\bc$ label at the end of the word $v$, the weight of this step is correctly $2\bd$.

\begin{figure}
\centering
\includegraphics[scale=.25]{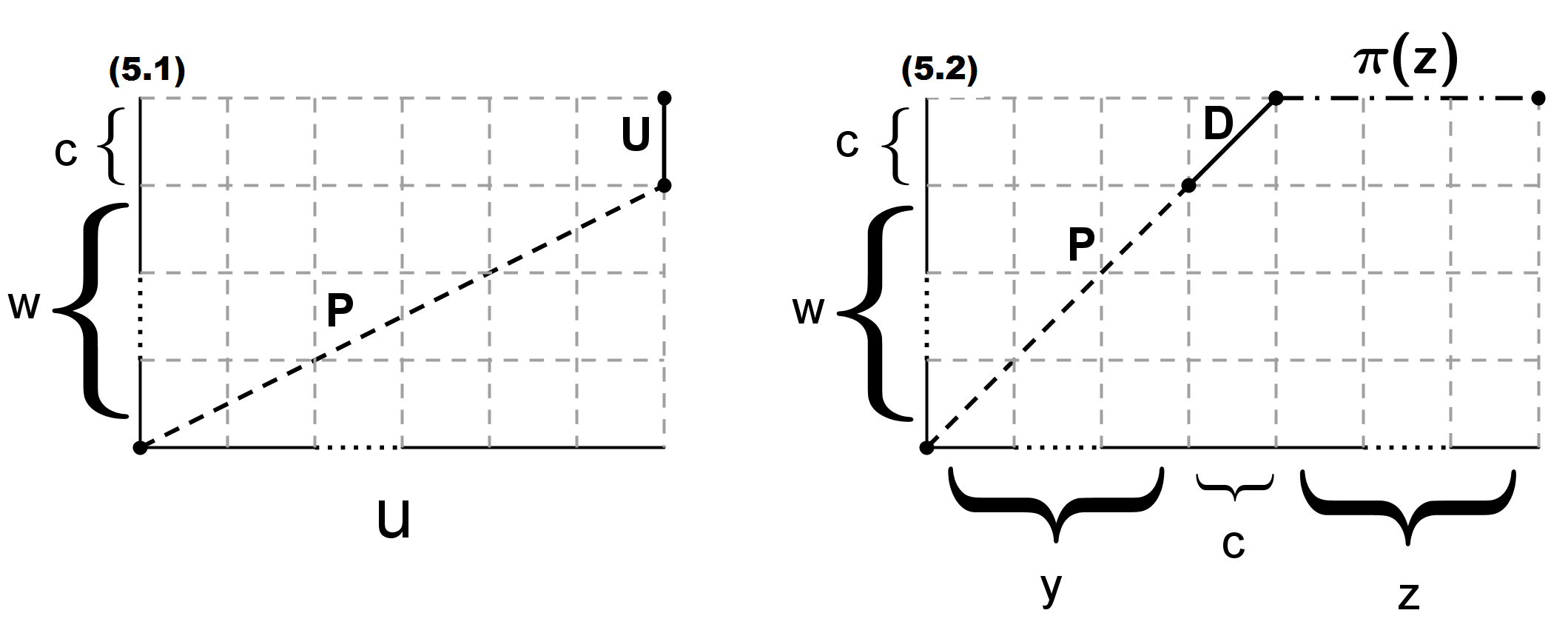}
\caption{: Illustrations of the lattice paths described in the first two subcases of Case 1}
\label{proof_example1}
\end{figure}

If $u$ is instead broken up at a $\bd$, we assume $u=y\cdot \bd\cdot z$; thus, $u$ splits as $y\otimes \bc\cdot z+y\cdot \bc \otimes z$. This leads to two terms, the first of which is
\begin{equation}
(y\diamond w) \cdot \bd\cdot \bc\cdot z=\sum_{P\in \Gamma(y,w)} \wt(P\cdot \D \cdot \R\cdot \pi(z)).
\label{c}
\end{equation}
Although the $\D$ step is above the first part of a $\bd$ label and to the right of a $\bc$ label, 1 is the correct coefficient of the weight of this $\D$ step since it is not followed by a $\U$, $\UU$, or $\D$ step.  The lattice paths described in equation~\eqref{c} and the following equation can be seen in Figure~\ref{proof_example2}.  

The other term we get is
\begin{equation}
(y\cdot \bc\diamond w)\cdot \bd\cdot z=\sum_{P'\in \Gamma(y\cdot \bc,w)}\wt(P')\cdot \bd\cdot z=\sum_{\substack{P\in \Gamma(y\cdot \bd,v)\\ \text{P ends with }\D}}\wt(P\cdot \pi(z)).
\label{d}
\end{equation}
First, note that the $\D$ step that is appended to $P'$ to create $P$ has the correct coefficient of $1$ since it is above the second half of a $\bd$ label.  As we switch labels from $y\cdot \bc$ to $y\cdot \bd$, it is important to notice that the coefficient of a $\D$ step above this $\bc$ label does not change.  The only scenario in which it could change is if it was to the right of the bottom of a $\bd$ label and was not followed by a $\U$ step, but this is impossible because a $\U$ step would be required to move vertically through the top half of the $\bd$ label.  

\begin{figure}
\centering
\includegraphics[scale=.25]{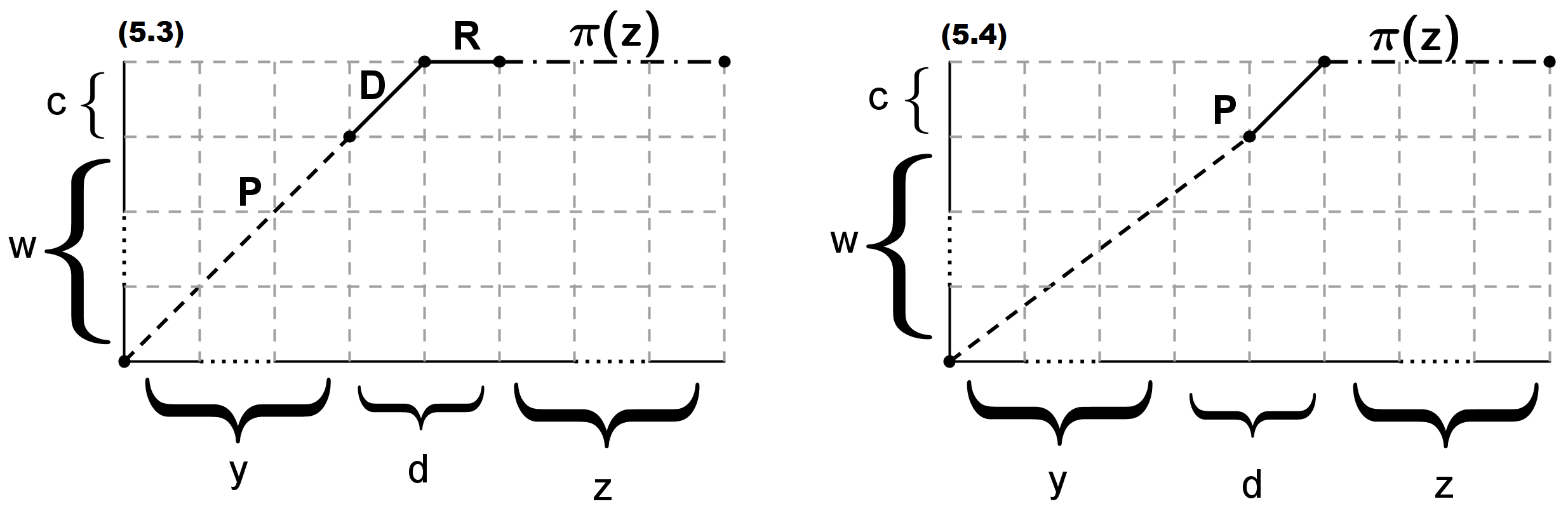}
\caption{: Illustrations of the lattice paths described in the last two subcases of Case 1}
\label{proof_example2}
\end{figure}

To avoid the subwords $\U\R$ and $\U\RR$, every lattice path in $\Gamma(u,w\cdot\bc)$ must either end in a $\U$ step or end in a $\D$ step followed by a horizontal path to the point $(p,q)$.  The paths within the three types of terms resulting from the coproduct cover all possible ways for this $\D$ step to occur, either above a~$\bc$ label or above one of the two parts of a $\bd$ label. Thus $\Gamma(u,w\cdot\bc)$ decomposes as the disjoint union
\begin{align}
\Gamma(u,w\cdot\bc)=\,&\{P\cdot \U: P\in \Gamma(u,w)\} \label{e}\\
&\dot\cup\, \{P\cdot \D\cdot \pi(z): P\in \Gamma(y,w), u=y\cdot \bc\cdot z\}\label{f}\\
&\dot\cup\, \{P\cdot \D\cdot \R\cdot \pi(z): P\in \Gamma(y,w),u=y\cdot \bd\cdot z\} \label{g}\\
&\dot\cup\, \{P\cdot \pi(z): P\in \Gamma(y\cdot \bd,v), P\text{ ends in }\D, u=y\cdot \bd\cdot z\},\label{h}
\end{align}
where the set~\eqref{e} is from equation~\eqref{a},~\eqref{f} from~\eqref{b},~\eqref{g} from~\eqref{c}, and~\eqref{h} from~\eqref{d}.
This concludes the proof for this case.

\textbf{Case 2}: Assume $v=w\cdot \bd$.  By applying equation~\eqref{diamond2}, we have
$$u\diamond(w\cdot \bd)=(u\diamond w)\cdot \bd +\sum_{u} (u_{(1)}\diamond w)\cdot \bd\cdot \pyr(u_{(2)}).$$

The first term, by induction, gives us 
\begin{equation}
(u\diamond w)\cdot \bd=\sum_{P\in \Gamma(u,w)} \wt(P\cdot \UU).
\label{i}
\end{equation}
See an illustration of the lattice paths in equation~\eqref{i} and the next equation in Figure~\ref{proof_example3}.

We once again separate the remaining terms from the coproduct depending on whether $u$ is broken up at a $\bc$ or $\bd$.  If broken up at a $\bc$, we assume $u=y\cdot \bc\cdot z$; hence, $u$ splits into $u_{(1)}\otimes u_{(2)}=2y\otimes z$ as it did in Case 1.  This gives the term 
\begin{equation}
\label{j}
(y\diamond w)\cdot 2\bd\cdot \pyr(z)=\sum_{\substack{P\in \Gamma(y,w), Q\in \Gamma(\bc\cdot z,\bd)\\ Q\text{ begins with }\D}}\wt(P\cdot Q).
\end{equation}
The $2\bd$ is the weight of the $\D$ step that it is above the $\bc$ label since it is to the right of the bottom of a $\bd$ label, so it remains to show that $\pyr(z)$ gives the weights of all of the remainders of the paths $Q$ after the $\D$ step.  Since this step is at the bottom part of a $\bd$ label on the vertical axis, rule 4 causes any path with $\D\R$ along any $\bd$ label to be invalid.  There also cannot be any path with a $\U$ step, except possibly as the final step.  Thus these paths only have horizontal steps with a $\U$ step at the end, or they only have horizontal steps with the exception of one $\D$ step, either above a $\bc$ label or following an $\R$ step along a $\bd$ label.

\begin{figure}
\centering
\includegraphics[scale=.25]{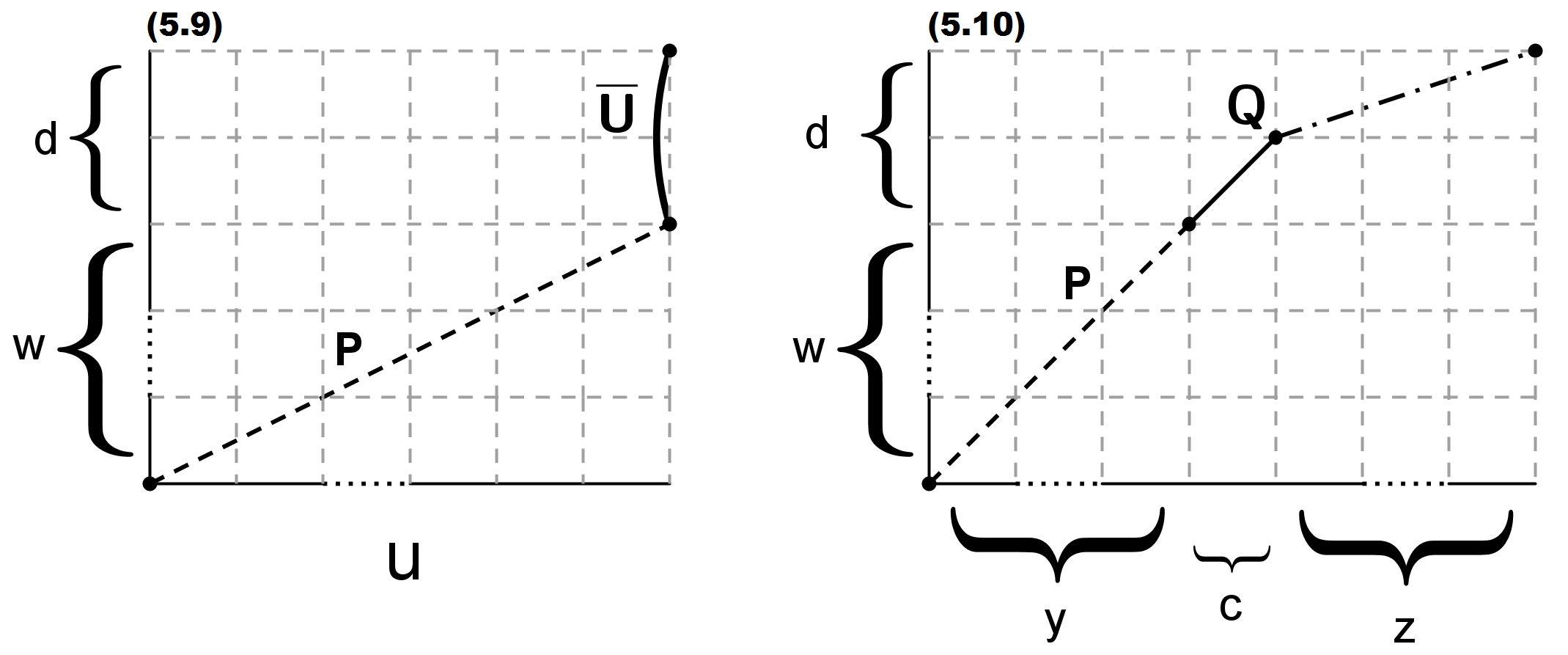}
\caption{: Illustrations of the lattice paths described in the first two subcases of Case 2}
\label{proof_example3}
\end{figure}

Recall that
$$\pyr(z)=z\cdot \bc+G(z).$$
The first term is 
$$z\cdot \bc=\wt(\pi(z)\cdot \U),$$
corresponding to the horizontal path with $\U$ appended to the end.

Since $G$ is a derivation, we apply the product rule to $z=z_1\cdots z_i$ to get $$G(z)= \sum_{j=1}^i z_1\cdots z_{j-1}\cdot G(z_j)\cdot z_{j+1}\cdots z_i.$$
If $z_j=\bc$, we have 
\begin{align*}
z_1\cdots z_{j-1}\cdot G(z_j)\cdot z_{j+1}\cdots z_i&= z_1\cdots z_{j-1} \cdot \bd \cdot z_{j+1}\cdots z_i\\
&= \wt(\pi(z_1\cdots z_{j-1})\cdot \D\cdot\pi(z_{j+1}\cdots z_i)),
\end{align*}
corresponding to the paths where the $\D$ step is above a $\bc$ label.  The weight of this step has coefficient~$1$ since it is along the top half of a $\bd$ label on the vertical axis.
On the other hand, if $z_j=\bd$, we have 
\begin{align*}
z_1\cdots z_{j-1}\cdot G(z_j)\cdot z_{j+1}\cdots z_i&= z_1\cdots z_{j-1} \cdot \bc\cdot\bd \cdot z_{j+1}\cdots z_i\\
&= \wt(\pi(z_1\cdots z_{j-1})\cdot \R\cdot \D\cdot \pi(z_{j+1}\cdots z_i)),
\end{align*}
corresponding to the paths with $\R\D$ steps above the $\bd$ label, where the coefficient of the weight of the $\D$ step is again $1$ by the same reasoning.  Therefore, $G(z)$ gives the correct paths that combine with the initial $\D$ step to make up the paths $Q$, proving equation~\eqref{j}.

If $u$ is broken up at a $\bd$, we assume $u=y\cdot \bd \cdot z$, and we have that $u$ splits as $y\otimes \bc\cdot z+y\cdot \bc \otimes z$. This gives two terms, the first being 
\begin{equation}
(y\diamond w)\cdot \bd \cdot \pyr(\bc\cdot z)=\sum_{\substack{P\in \Gamma(y,w), Q\in \Gamma(\bd\cdot z,\bd)\\ Q\text{ begins with }\D}}\wt(P\cdot Q).
\label{k}
\end{equation}
The $\bd$ is the correct weight of the first $\D$ step in $Q$ since it cannot be followed by a $\U$ step.  Otherwise, the path would be invalid since it would have $\U\R$ or $\U\RR$ as a subword.  $\pyr(\bc\cdot z)$ gives the weights of the remainders of the paths $Q$ due to an argument analogous to the one used in the previous subcase, because treating the second half of the $\bd$ label on the horizontal axis as a $\bc$ label does not change any of the weights of these paths.  Illustrations of the lattice paths in equation~\eqref{k} and the following equation can be found in Figure~\ref{proof_example4}.

\begin{figure}
\centering
\includegraphics[scale=.25]{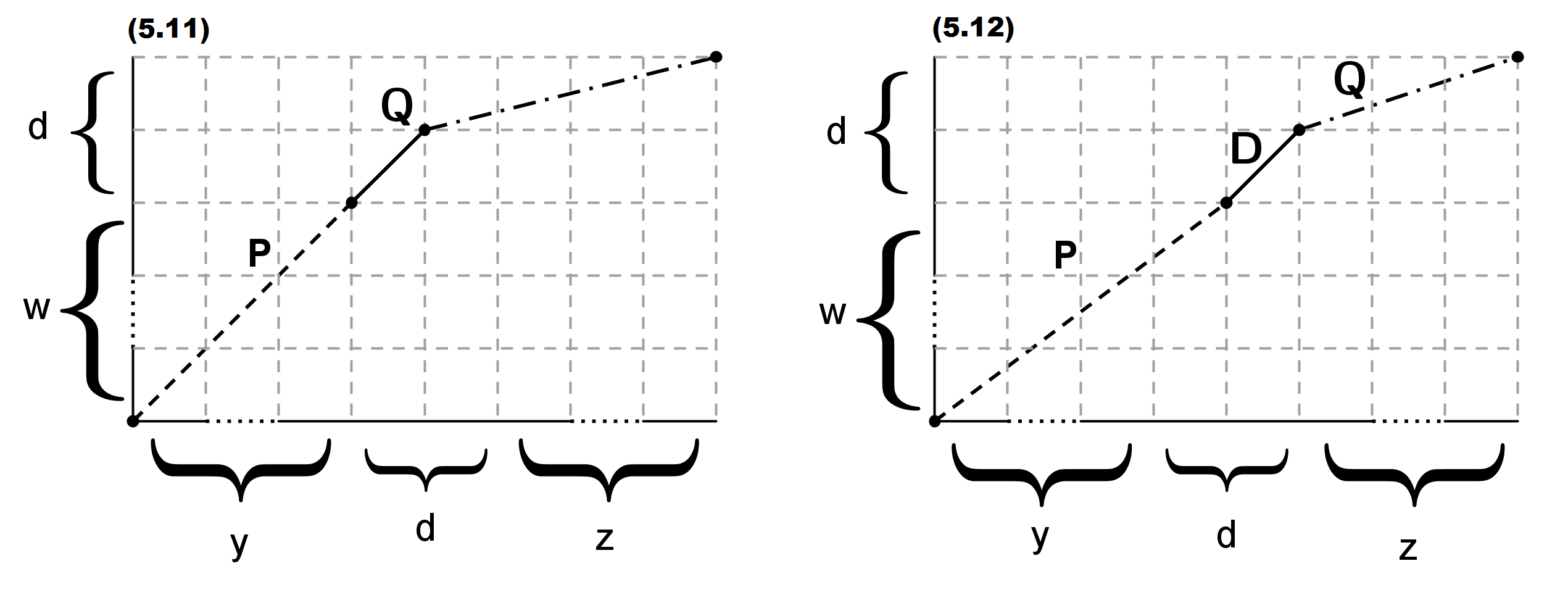}
\caption{: Illustrations of the lattice paths described in the last two subcases of Case 2}
\label{proof_example4}
\end{figure}

The second term from this situation is 
\begin{align}
(y\cdot \bc \diamond w)\cdot \bd\cdot \pyr(z)&=\left(\sum_{P'\in \Gamma(y\cdot\bc,w)} \wt(P')\right) \cdot \bd \cdot \pyr(z) \nonumber\\
&=\sum_{\substack{P\in \Gamma(i+1,q-2), Q\in\Gamma(p-i-2,1)\\ P\cdot \D \cdot Q\in \Gamma(y\cdot\bd\cdot z,w\cdot\bd)}} \wt(P\cdot \D\cdot Q).
\label{l}
\end{align}
Here, we are assuming the degree of $y$ is $i$; hence, the degree of $z$ is $p-i-2$.  Note that the path $P'$ does not have its weight changed as it becomes the path $P$ when the $\bc$ label is switched to become the first half of a $\bd$ label.  This is true since the only possible difference could be the coefficient of a~$\D$ step above the final $\bc$ label.  However, this coefficient will not change since it must be followed by a~$\U$ or $\UU$ step if the $\D$ step is not the final step in $P'$, or it is followed by a $\D$ step if it is the final step in $P'$.  The coefficient of $1$ is correct for the $\D$ step between the paths $P$ and $Q$ since it is above the second part of a $\bd$ label.  Although it is not possible to partition the labels in order to have the correct weights when writing $P$ and $Q$ as elements of $\Gamma(x,x')$ for some $\bc\bd$-monomials $x$~and~$x'$ as was done in the previous cases, it is still clear that the contribution that $Q$ makes to the weight is $\pyr(z)$, similarly to the last two subcases.

The lattice paths in $\Gamma(u,w\cdot \bd)$ must either end in a $\UU$ step, or by rule 1, there must be two $\D$ steps to the right of the last $\bd$ label of $v=w\cdot \bd$ with horizontal paths between and after these steps.  The three types of terms from the coproduct consist of all ways for these $\D$ steps to occur, with the three types being distinguished by whether the first $\D$ step is above a $\bc$ label, the first part of a $\bd$ label, or the second part of a $\bd$ label.  Therefore, $\Gamma(u,w\cdot \bd)$ decomposes as the disjoint union
\vspace{.1cm}\begin{align}
\Gamma(u,w\cdot \bd)&=\,\{P\cdot \overline{\bf{U}} : P\in \Gamma(u,w)\} \label{m}\\
&\dot\cup \,\{P\cdot Q: P\in \Gamma(y,w), Q\in \Gamma(\bc \cdot z, \bd), Q\text{ begins with }\D, u=y\cdot \bc\cdot z\} \label{n}\\
&\dot\cup \, \{P\cdot Q: P\in \Gamma(y,w),Q \in \Gamma(\bd \cdot z,\bd), Q\text{ begins with } \D, u=y\cdot \bd\cdot z\} \label{o}\\
&\dot\cup  \,\{P\cdot \D\cdot Q: P\in \Gamma(i+1,q-2), Q\in \Gamma(p-i-2,1),\label{p}\\
&\hspace{.5cm} P\cdot \D\cdot Q\in \Gamma(y\cdot\bd\cdot z,w\cdot\bd), u=y\cdot \bd\cdot z\}, \nonumber
\end{align}\vspace{.1cm}where the set~\eqref{m} is from equation~\eqref{i},~\eqref{n} from~\eqref{j},~\eqref{o} from~\eqref{k}, and~\eqref{p} from~\eqref{l}.  This decomposition gives us the proof for the case of $v$ ending in a $\bd$, concluding the proof of the theorem.
\end{proof}


\section{Concluding Remarks}
The effect on the $\bc\bd$-index of a second important operation on posets was studied in \cite{Ehrenborg_Fox} and \cite{Ehrenborg_Readdy}.  This operation is the Cartesian product of posets, defined at the beginning of Section 3.  As the diamond product of posets is related to the Cartesian product of polytopes, the Cartesian product of posets is connected to the \textit{free join} of polytopes, defined as follows. If $V$ is an $m$-dimensional polytope  and $W$ is an $n$-dimensional polytope, then embed $V$ and $W$ in $\mathbb{R}^{m+n+1}$ by 
$$V'=\{(x_1,\ldots,x_m,\underbrace{0\ldots,0}_n,0)\in\mathbb{R}^{m+n+1}: (x_1,\ldots x_m)\in V\}$$
and likewise by
$$W'=\{(\underbrace{0,\ldots, 0}_m,x_1,\ldots,x_n,1)\in \mathbb{R}^{m+n+1}:(x_1,\ldots,x_n)\in W\}.$$
Then the free join $V\ovee W$ is the $(m+n+1)$-dimensional polytope defined as the convex hull of~$V'$~and~$W'$.  Kalai \cite{Kalai} observed that the face lattice of the free join of two polytopes is the Cartesian product of the two face lattices, i.e., for two polytopes $V$ and $W$ we have $\mathcal{L}(V\ovee W)=\mathcal{L}(V)\times \mathcal{L}(W)$. Ehrenborg and Readdy \cite{Ehrenborg_Readdy} developed a bilinear operator from $\zab\times \zab$ to $\zab$, called the mixing operator~$M$, in order to study the $\bc\bd$-index of the Cartesian product of posets, or likewise the $\bc\bd$-index of the free join of polytopes.  As with the diamond product operator, Section 6 of \cite{Ehrenborg_Fox} and Section~10 of \cite{Ehrenborg_Readdy} give the definition and recurrences for this operator.  The recurrence is nearly identical to that of the diamond product; however, differing initial conditions cause the degree of $M(u,v)$ to be one higher than the degree of $u\diamond v$.  Is there a similar lattice path interpretation for this product?  Even a good interpretation for the easier cases of $\bc^m \times \bc^n$ or  the Cartesian product of $\ba\bb$-monomials is currently unknown. 

Recently Carl Lee (personal communication) found an equation that relates the free join and Cartesian product of polytopes, while also involving the pyramid and prism operations.  Together with Ehrenborg, the author used a chain counting argument to show it is true for $\bc\bd$-indices of the analogous operations on posets.  It states that for two posets $P$ and $Q$, we have
$$\Psi(P\times Q)=\Psi(\pyr(P)\diamond Q)+\Psi(P\diamond \pyr(Q))-\Psi(\prism(P\diamond Q)).$$
If one could develop lattice path interpretations for the three simpler terms on the right hand side, it would allow us to have an interpretation for the Cartesian product $P\times Q$.

A different approach to studying how flag $f$-vectors
change during poset operations such as the Cartesian product and diamond product is by using quasi-symmetric functions.
The quasi-symmetric function of a poset is
multiplicative with respect to Cartesian product;
see~\cite[Proposition~4.4]{Ehrenborg_Hopf}.
Similarly, the type $B$
quasi-symmetric function of a poset is
multiplicative with respect to the diamond product;
see~\cite[Theorem~13.3]{Ehrenborg_Readdy_Tchebyshev}.  Could this approach be helpful in gaining a better understanding of these product operators?

\section*{Acknowledgements}
The author would like to thank Richard Ehrenborg for reading earlier versions of this paper. The author was partially supported by
National Security Agency grant~H98230-13-1-0280.

\newcommand{\journal}[6]{{#1,} #2, {\it #3} {\bf #4} (#5) #6.}
\newcommand{\dissertation}[4]{{#1,} #2, {\it #3,} (#4)}
\newcommand{\book}[5]{{#1,} {\it #2,} #3, #4.}

\end{document}